\documentclass[12pt,amscd]{amsart}
\usepackage[all]{xy}
\usepackage{graphicx}
\usepackage{amsmath,amsxtra,amssymb,latexsym, amscd,amsthm}

\usepackage{tikz}
\usepackage{ytableau}

 \usepackage{indentfirst}
\usepackage[mathscr]{eucal}
  \usepackage[pagebackref=true]{hyperref}

\newtheorem{thm}{Theorem}[section]
\newtheorem{cor}[thm]{Corollary}
\newtheorem{lem}[thm]{Lemma}
\newtheorem{prop}[thm]{Proposition}
\theoremstyle{definition}
\newtheorem{defn}[thm]{Definition}
\newtheorem{exam}[thm]{Example}

\newtheorem{rem}[thm]{Remark}

\numberwithin{equation}{section}

\DeclareMathOperator{\NN}{\mathbb {N}}
\DeclareMathOperator{\ZZ}{\mathbb {Z}}

\DeclareMathOperator{\lk}{lk}

\DeclareMathOperator{\height}{height}

\DeclareMathOperator{\depth}{depth}
\DeclareMathOperator{\pd}{pd}
\DeclareMathOperator{\supp}{supp}

\DeclareMathOperator{\reg}{reg}

\def\D {\Delta}

\def\M {\mathcal M}

\def\a {\mathbf a}
\def\b {\mathbf b}

\def\m {\mathfrak m}
\def\F {\mathfrak F}

\def\k {\mathrm{k}}
\def\h {\widetilde{H}}
\def\c {\mathbf{c}}
\def\ww {\mathbf{w}}

\begin{document}

\title{Depth and regularity of tableau ideals}

\author{Do Trong Hoang}
\address{School of Applied Mathematics and Informatics, Hanoi University of Science and Technology, 1 Dai Co Viet, Hai Ba Trung, Hanoi, Vietnam.}

\email{hoang.dotrong@hust.edu.vn}

\author{Thanh Vu}
\address{Institute of Mathematics, VAST, 18 Hoang Quoc Viet, Hanoi, Vietnam}
\email{vuqthanh@gmail.com}


\date{}

\dedicatory{Dedicated to Professor Brian Harbourne on the occasion of his 65th birthday}
\commby{}
\maketitle
\begin{abstract}
   We compute the depth and regularity of ideals associated with arbitrary fillings of positive integers to a Young diagram, called the tableau ideals.
\end{abstract}

\maketitle

\section{Introduction}
\label{sect_intro}
A Young diagram, also called a Ferrers diagram, is a collection of boxes arranged in left-justified rows with row lengths weakly decreasing. Any way of putting a positive integer in each box of a Young diagram will be called  a {\it filling}  of the diagram. A filling $Y$ is called weakly increasing if it is weakly increasing across each row and down each column of the diagram. (In standard terminology, a semistandard Young tableau requires strictly increasing down each column of the diagram). The shape of a tableau $Y$ is the corresponding partition $\lambda$ whose $\lambda_i$ is the number of boxes in the $i$th row of $Y$. Young tableaux play an important role in reprensentation theory of symmetric groups and general linear groups, as well as Schubert calculus. They also have wide applications in other fields (see \cite{Ful} for more detail).

For each Young diagram $\lambda$, Corso and Nagel \cite{CN} introduced and studied the so-called Ferrers ideal, $I_\lambda$. They described its minimal free resolution. To each filling $Y$, we also have a natural associated monomial ideal, which generalizes the Ferrers ideal. More precisely, assume that $\lambda = (\lambda_1,\ldots, \lambda_n)$ and $Y$ is the filling of $\lambda$ whose value at the $ij$ box is $w(i,j)$. The {\it tableau ideal} associated with $Y$ is the following ideal inside a standard graded polynomial ring $S = \k[x_1, \ldots, x_n,y_1,\ldots,y_m]$ over a field $\k$, where $m=\lambda_1$
$$I(Y) = \left ( (x_iy_j)^{w(i,j)} \mid 1\le i\le n, 1 \le j \le \lambda_i \right ).$$
The tableau ideal is the edge ideal of an edge-weighted graph introduced by Paulsen and Sather-Wagstaff \cite{PS}. Recently, there has been a surge of interest in characterizing weights for which the edge ideals of edge-weighted graphs are Cohen-Macaulay \cite{FSTY, DMV, Hi, W}. Nonetheless, computing the exact value of the depth and regularity of edge ideals of edge-weighted graphs is more subtle. In this paper, we give a simple recursive formula for the depth and regularity of $I(Y)$ for arbitrary fillings $Y$. Our work is the first result to compute the depth and regularity of edge ideals of a family of edge-weighted graphs for arbitrary weights.

To describe our main results, we first introduce some notation. Let 
$$\omega = \min \left \{ w(i,j) \mid 1 \le i \le n, 1 \le j \le \lambda_i  \right \}$$
be the minimum weight. A box $(\gamma,\delta)$ is called a {\it minimal box} of $Y$ if $w(\gamma,\delta) = \omega$ and $w(i,j) > \omega$ for all $(i,j) \neq (\gamma,\delta)$  such that $i \le \gamma$ and $j \le \delta$. We have

\begin{thm}\label{thm_depth_tableau_0}
    Let $Y$ be an arbitrary filling of a Young diagram $\lambda$. Let $(\gamma,\delta)$ be a minimal box of $Y$. Then 
    $$\depth S/I(Y) = \min \left \{ \depth (S/(I(Y) + (x_\gamma))), \depth (S/(I(Y) + (y_\delta))) \right \}.$$
\end{thm}

Consequently, we deduce that the depth of $I(Y)$ is the same as that of the corresponding Ferrers ideal when $Y$ is weakly increasing and classify all weights for which $I(Y)$ is Cohen-Macaulay.

\begin{cor}\label{cor_weakly_increasing_depth} Assume that $Y$ is a weakly increasing filling of a Young diagram $\lambda$. Then $$\depth S/I(Y) = \depth S/I_\lambda = m+n - \max \{ \lambda_i + i -1 \mid i = 1, \ldots, n\}.$$ 
 \end{cor}

 \begin{cor}
     Let $Y$ be an arbitrary filling of a Young diagram $\lambda$. Then $I(Y)$ is Cohen-Macaulay if and only if $\lambda = (n,n-1,\ldots,2,1)$ and $Y$ is weakly increasing.
 \end{cor}

We also have a similar formula for the regularity of $I(Y)$.

\begin{thm}\label{thm_regularity_tableau_0} Let $Y$ be an arbitrary filling of a Young diagram $\lambda$ with at least two boxes. Let $(\gamma,\delta)$ be a minimal box of $Y$. Then, 
$$\reg S/I(Y) = \omega-1 + \max \left \{ \reg S/(I(Y) + (x_{\gamma})), \reg S/(I(Y) + (y_{\delta})) \right \}.$$   
\end{thm}

Note that $I(Y) + (x_{\gamma})$ and $I(Y) + (y_\delta)$ correspond to the tableau ideals obtained by deleting the $\gamma$th row and $\delta$th column of $Y$ respectively. Thus, Theorem \ref{thm_depth_tableau_0} and Theorem \ref{thm_regularity_tableau_0} give us simple recursive formulae for computing the depth and regularity of $I(Y)$. Here is an example to illustrate our result.

\begin{exam}
    Consider the following filling of a Young diagram 
    \begin{center}
\begin{tabular}{ccc}
$\begin{ytableau}
*(white) 3 & *(red) 1 & *(white) 5 & *(blue) 6\\
*(yellow) 2 & *(white) 3 & *(green) 4 & *(orange) 6   \\
*(white) 2 & *(white) 3 & *(white) 5    \\
*(white) 2 & *(white) 4 \\
*(white) 3
\end{ytableau}$\\[5pt]
\end{tabular}
 \end{center}
The tableau ideal associated with $Y$ is 
\begin{align*}
    I(Y) = ( &(x_1y_1)^3,x_1y_2,(x_1y_3)^5,(x_1y_4)^6,(x_2y_1)^2,(x_2y_2)^3,(x_2y_3)^4,(x_2y_4)^6 \\
    & (x_3y_1)^2,(x_3y_2)^3,(x_3y_3)^5,(x_4y_1)^2,(x_4y_2)^4,(x_5y_1)^3 ).
\end{align*}
First, the red box is the one with the smallest weight. Hence, we have 
$$\depth S/I(Y) = \min \{ \depth S/(I(Y) + (x_1)), \depth S/(I(Y) + (y_2)) \}.$$
After deleting either the first row or the second column, the yellow box will become the minimal box. Repeating the process, we get $\depth S/I(Y) = 2$. 

For the regularity, we see that deleting the first row will give us higher regularity. Repeating the process, we choose the yellow, green, blue, and orange boxes and get
$$\reg S/I(Y) = 1 + 3 + 5 + 5 + 6 = 20.$$
\end{exam}

We now describe the idea of proofs of the main results. For a monomial ideal $I$, Hochster \cite{Hoc} introduced the set of {\it associated radicals} of $I$, namely $\sqrt{I:u}$ for monomials $u\notin I$ and proved that 
$$\depth S/I = \min \{\depth S/J \mid J \text{ is an associated radical of } I \}.$$
In \cite{MNPTV}, Minh, Nam, Phong, Thuy, and Vu proved that the regularity of $I$ can also be computed via the associated radicals of $I$. Thus, a key step in proving our main results is to understand the associated radicals of tableau ideals. More precisely, let $Y'$ be the filling of $\lambda$ obtained by replacing the weight $\omega$ on the $(\gamma,\delta)$ box of $Y$ by $1$. The steps to prove Theorem \ref{thm_depth_tableau_0} are 
\begin{enumerate}
    \item By comparing the associated radicals of $I(Y)$ and $I(Y')$, we prove in Lemma \ref{lem_depth_reduction_w_1} that $\depth S/I(Y) \ge \depth S/I(Y')$. 
    \item In Lemma \ref{lem_depth_upperbound_add_variable}, we prove that $\depth S/I(Y) \le \depth S/(I(Y) + (x_\gamma))$ by using the structure of Ferrers ideals to reduce to the case where $I(Y) + (x_\gamma)$ has an associated radical corresponding to a complete bipartite graph (i.e., a Ferrers ideal of a rectangular partition). 
    \item By using the result of Caviglia, Ha, Herzog, Kummini, Terai, and Trung \cite{CHHKTT} stating that $\depth S/I \in \{\depth S/(I,x), \depth (S/I:x)\}$, we deduce the conclusion.
\end{enumerate}
Similarly, the steps to prove Theorem \ref{thm_regularity_tableau_0} are
\begin{enumerate}
    \item By analyzing the structure of the degree complexes of $I(Y)$, in Lemma \ref{lem_reg_reduction_to_weight_1} we prove that $\reg S/I(Y) \le \reg S/I(Y') + (\omega - 1)$. 
    \item Similarly, by analyzing the degree complexes of $I(Y)$ and $I(Y) + (x_\gamma)$, in Lemma \ref{lem_reg_lower_bound}, we prove that $\reg S/I(Y) \ge (\omega - 1) + \reg S/(I(Y) + (x_\gamma))$. 
    \item By using the result of Dao, Huneke, and Schweig \cite{DHS} stating that $\reg S/I  \in \{ \reg S/(I,x), \reg S/I:x + 1\}$, we deduce the conclusion.
\end{enumerate}
While the steps to establish the formulae for the depth and regularity of tableau ideals are similar, the proofs are on different extremes. The formulae in Theorem \ref{thm_depth_tableau_0} and Theorem \ref{thm_regularity_tableau_0} can be restated in terms of combinatorics of the tableau $Y$ as follows. Each box $(i,j)$ of $Y$ is marked with either $r$ or $c$, corresponding to the $i$th row or the $j$th column of $Y$. A marked $(i,j)$ box is denoted by $(i,j,m)$ with $m \in \{r,c\}$. For a marked box $(i,j,m)$, we let
$$L(i,j,m) = \begin{cases}
    \text{the } i \text{th row of } Y & \text{ if } m = r,\\
    \text{the } j \text{th column of } Y & \text{ if } m = c.
\end{cases}$$
A collection of marked boxes $M = \{ (i_1,j_1,m_1), \ldots, (i_s,j_s,m_s)\}$ is called an {\it admissible collection} if $Y = L(i_1,j_1,m_1) \cup \cdots \cup L(i_s,j_s,m_s)$ and $(i_t,j_t)$ is a minimal box of the tableau $Y_t$ obtained by deleting $L(i_1,j_1,m_1), \ldots, L(i_{t-1},j_{t-1},m_{t-1})$. The shape of $Y_t$ is denoted by $\lambda^t$. Its conjugate partition is denoted by $\mu^t$. For each $t = 1, \ldots, s$, we set 
$$d_t = \begin{cases} (\lambda^t)_1 - (\lambda^t)_2 & \text{ if } m_t = c \text{ and } L(i_t,j_t,m_t) \text{ is the first column of } Y_t,\\
(\mu^t)_1 - (\mu^t)_2 & \text{ if } m_t = r \text{ and } L(i_t,j_t,m_t) \text{ is the first row of } Y_t, \\
0 & \text{ otherwise}.
\end{cases}$$
The depth and regularity of $M$ with respect to $Y$ are defined by \begin{align*}
    d(M,Y) & = d_1 + \cdots + d_s,\\
    r(M,Y) &= \sum_{i=1}^s(w(i_1,j_1)-1) + w(i_s,j_s).
\end{align*}
Then we have 
\begin{thm}\label{thm_main}
    Let $Y$ be an arbitrary filling of a Young diagram $\lambda$. Denote by $\M$ the set of all admissible collections of marked boxes of $Y$. Then 
    \begin{align*}
        \depth S/I(Y) &= \min\{ d(M,Y) \mid M \in \M\}, \\
        \reg S/I(Y) &= \max \{ r(M,Y) \mid M \in \M\}.
    \end{align*}
\end{thm}

We now describe the organization of the paper. In Section \ref{sec_pre}, we first state definitions and standard results about the depth and regularity of monomial ideals. Next, we recall the notion of edge ideals of edge-weighted graphs. Finally, we consider Ferrers ideals and tableau ideals. In Section \ref{sec_depth_tableaux_ideal}, we compute the depth of tableau ideals. In Section \ref{sec_regularity_tableaux_ideal}, we compute the regularity of tableau ideals and prove Theorem \ref{thm_main}.

\section{Preliminaries}\label{sec_pre}
In this section, we first state definitions and some properties concerning the depth and regularity of monomial ideals. Next, we recall edge ideals of edge-weighted graphs and prove some results about their associated radicals. We then recall the results of Corso and Nagel \cite{CN} about Ferrers ideals. Finally, we introduce tableau ideals and establish some properties of their associated radicals.

In the first three subsections, we let $S = \k[x_1,\ldots, x_n]$ be a standard graded polynomial ring over a field $\k$ with the maximal homogeneous ideal $\m = (x_1,\ldots, x_n)$. We first recall the notion of depth and regularity. For more information, we refer to \cite{BH}.

\subsection{Depth and regularity} For a finitely generated graded $S$-module $L$, the depth of $L$ is defined to be
$$\depth(L) = \min\{i \mid H_{\m}^i(L) \ne 0\},$$
where $H^{i}_{\m}(L)$ denotes the $i$-th local cohomology module of $L$ with respect to $\m$. 

The regularity of $L$ is defined to be 
$$\reg (L) = \max \{i + j \mid H_{\m}^i(L)_j \neq 0\}.$$

\begin{defn} A finitely generated graded $S$-module $L$ is called Cohen-Macaulay if $\depth L = \dim L$. A homogeneous ideal $I \subseteq S$ is said to be Cohen-Macaulay if $S/I$ is a Cohen-Macaulay $S$-module.
\end{defn}

We have the following result of Hochster \cite{Hoc} on the depth of monomial ideals. 
\begin{thm}\label{thm_Hochster_depth} Let $I$ be a monomial ideal. Then
    $$\depth S/I = \min \{ \depth S/\sqrt{I:f} \mid f \text{ is a monomial such that } f \notin I\}.$$
\end{thm}

We also have the following results on the depth and regularity of monomial ideals modulo a variable.

\begin{lem}\label{lem_depth_reg_colon} Let $I$ be a monomial ideal and $x$ a variable. Then 
\begin{align*}
    \depth S/I & \in \{\depth (S/I:x), \depth (S/(I,x))\} \\
    \reg S/I & \in \{ \reg (S/I:x) + 1, \reg (S/(I,x)) \}.
\end{align*}  
\end{lem}
\begin{proof}
    This is \cite[Corollary 3.3]{CHHKTT}.
\end{proof}

\begin{lem}\label{lem_depth_reg_adding_variable} Let $I$ be a monomial ideal and $x$ a variable. Then 
\begin{enumerate}
    \item $\depth S/(I,x) \ge \depth S/I - 1,$
    \item $\reg S/I \ge \reg S/(I,x)$.
\end{enumerate}
\end{lem}
\begin{proof} For the first statement, let $J$ be the restriction of $I$ to the polynomial subring $T$ on the variable $\{x_1,\ldots,x_n\} \setminus \{x\}$. By \cite[Lemma 4.4]{HHZ}, $\pd (T/J) \le \pd (S/I)$. The conclusion follows from the Auslander-Buchsbaum formula. 

    The second statement follows from \cite[Corollary 4.8]{CHHKTT}.
\end{proof}
In the ideals of the form $I+(x)$ and $\sqrt{I:f}$, some variables will be part of the minimal generators, and some will not appear in any of the minimal generators. A variable that does not divide any minimal generators of a monomial ideal $J$ will be called a free variable of $J$. We have 

 \begin{lem}\label{lem_depth_variables_clearing} Assume that $I = J + (x_a, \ldots, x_b)$ and $x_{b+1}, \ldots, x_n$ are free variables of $J$ where $J$ is a monomial ideal in $R = \k[x_1,\ldots,x_{a-1}]$. Then
 \begin{align*}
  \depth S/I & = \depth R/J + (n-b), \\
  \reg S/I & = \reg R/J.   
 \end{align*}
 \end{lem} 
\begin{proof}
    Since free variables do not contribute to the minimal free resolution of $I$, we have $\pd(S/I) = \pd (T/I)$ and $\reg S/I = \reg T/I$ where $T = \k[x_1,\ldots,x_b]$. Now, $I$ is the mixed sum of $J$ and a linear ideal. The conclusion follows from the Auslander-Buchsbaum formula and \cite[Theorem 1.1]{NV}. 
\end{proof}
To compute the regularity of monomial ideals, we use the following result of Minh, Nam, Phong, Thuy, and Vu \cite{MNPTV}. We first introduce some notation. For an exponent $\a = (a_1, \ldots, a_n) \in \NN^n$, we set $x^\a = x_1^{a_1} \cdots x_n^{a_n}$ and $|\a| = a_1 + \cdots + a_n$. The support of $\a$ is 
$$\supp \a = \{ i \in [n] \mid a_i \neq 0\}.$$

\begin{lem}\label{Key0}
Let $I$ be a monomial ideal in $S$. Then
\begin{multline*}
\reg(S/I)=\max\{|\a|+i \mid \a\in\NN^n,i\ge 0,\h_{i-1}(\lk_{\D_\a(I)}F;\k)\ne 0\\ \text{ for some $F\in \D_\a(I)$ with $F\cap \supp \a=\emptyset$}\},
\end{multline*}
where $\D_\a(I) = \Delta(\sqrt{I:x^\a})$ is the degree complex of $I$ in degree $\a$. 
\end{lem}
\begin{proof}
    Follows from Lemma 2.12 and Lemma 2.19 of \cite{MNPTV}.
\end{proof}
Note that $\Delta(J)$ is the Stanley-Reisner complex corresponding to a squarefree monomial ideal $J$ (see the next section for more detail). 

\begin{defn} A pair $(\a,i) \in \NN^n \times \NN$ is called a {\it critical pair} of $I$ if there exists a face $F$ of $\D_\a(I)$ with $F \cap \supp \a = \emptyset$ and an index $i$ such that $\h_{i-1}(\lk_{\D_\a(I)} F;\k) \neq 0$. It is called an {\it extremal pair} of $I$ if furthermore, $\reg S/I = |\a| + i$.    

An exponent $\a \in \NN^n$ is called a {\it critical exponent} (respectively an {\it extremal exponent}) of $I$ if $(\a,i)$ is a critical pair (respectively an extremal pair) of $I$ for some $i \ge 0$.
\end{defn}
When $(\a,i)$ is a critical pair (respectively extremal pair) of $I$, we also call $(x^\a,i)$ a critical pair (respectively extremal pair) of $I$. 

For a monomial $f$ in S and $i \in [n] = \{1, \ldots, n\}$, $\deg_i (f)  = \max (t \mid  x_i^t \text{ divides } f)$ denotes the degree of $x_i$ in $f$. For a monomial ideal $I$, $\rho_i(I)$ is defined by
$$ \rho_i(I) = \max (\deg_i (u) \mid u \text{ is a minimal monomial generator of } I).$$

By \cite[Remark 2.13]{MNPTV}, we have
\begin{rem}\label{rem_critical_sets} A critical exponent of a monomial ideal $I$ belongs to the finite set
$$\Gamma (I) = \{\a \in \NN^n \mid a_j < \rho_j(I) \text{ for all } j = 1, \ldots, n\}.$$
\end{rem}

We now demonstrate a situation where one can reduce the computation of the regularity of a monomial ideal $I$ by adding a variable.

\begin{lem}\label{lem_regularity_equality_adding_variable} Let $I$ be a monomial ideal. Assume that $(\a,i)$ is an extremal pair of $I$ and $F$ is a face of $\Delta_\a(I)$ such that $F \cap \supp \a = \emptyset$ and $\h_{i-1}(\lk_{\Delta_\a(I)} F;\k) \neq 0$. Let $x$ be a variable in $F$. Then $\reg S/I = \reg S/(I,x)$.    
\end{lem}
\begin{proof}
    Let $J = I + (x)$. Since $x\in F$, $x \notin \supp \a$. By \cite[Lemma 2.24]{MNPTV}, $\sqrt{J:x^\a} = \sqrt{I:x^\a} + (x)$. In other words, $\Delta_\a(J)$ is the restriction of $\Delta_\a(I)$ to $[n]\setminus \{x\}$. Let $F' = F \setminus \{x\}$. Then, $\lk_{\Delta_\a(I)} F = \lk_{\Delta_\a(J)} F'$. Thus, $(\a,i)$ is a critical pair of $J$. By Lemma \ref{lem_depth_reg_adding_variable} and Lemma \ref{Key0}, 
    $$\reg S/J \ge |\a| + i = \reg S/I \ge \reg S/J.$$ 
    The conclusion follows.
\end{proof}

Since we will compute the depth and regularity of monomial ideals via their associated radicals, we now introduce the notion of simplicial complexes and Stanley-Reisner correspondence.

\subsection{Simplicial complexes and Stanley-Reisner correspondence} 

Let $\Delta$ be a simplicial complex on $[n]=\{1,\ldots, n\}$ that is a collection of subsets of $[n]$ closed under taking subsets. The set of its maximal elements under inclusion, called facets, is denoted by $\F(\Delta)$.

A simplicial complex $\D$ is called a cone over $x\in [n]$ if $x\in B$ for any $B\in \F(\Delta)$.

For a face $F\in\Delta$, the link of $F$ in $\Delta$ is the subsimplicial complex of $\Delta$ defined by
$$\lk_{\Delta}F=\{G\in\Delta \mid  F\cup G\in\Delta, F\cap G=\emptyset\}.$$

For each subset $F$ of $[n]$, let $x_F=\prod_{i\in F}x_i$ be a squarefree monomial in $S$. We now recall the Stanley-Reisner correspondence.

\begin{defn}For a squarefree monomial ideal $I$, the Stanley-Reisner complex of $I$ is defined by
$$ \Delta(I) = \{ F \subseteq [n] \mid x_F \notin I\}.$$

For a simplicial complex $\Delta$, the Stanley-Reisner ideal of $\Delta$ is defined by
$$I_\Delta = (x_F \mid  F \notin \Delta).$$
The Stanley-Reisner ring of $\Delta$ is $\k[\Delta] =  S/I_\Delta.$
\end{defn}

\begin{defn} The $q$-th reduced homology group of $\Delta$ with coefficients over $\k$, denoted $\h_q(\Delta; \k)$ is defined to be the $q$-th homology group of the augmented oriented chain complex of $\Delta$ over $\k$.
\end{defn}
A simplicial complex $\D$ is called {\it acyclic} if $\h_i(\Delta;\k) = 0$ for all $i$.
\begin{rem} Let $\Delta$ be a simplicial complex. Then \begin{enumerate}
    \item $\h_{-1}(\Delta;\k) \neq 0$ if and only if $\Delta$ is the empty complex (i.e., $\Delta=\{\emptyset\}$).
    \item If $\Delta$ is a cone over some $t \in [n]$ or $\Delta$ is the void complex  (i.e., $\Delta=\emptyset$), then it is acyclic.
\end{enumerate}
\end{rem}

\subsection{Edge ideals of edge-weighted graphs and their associated radicals} 

Let $G$ denote a finite simple graph over the vertex set $V(G) = [n] = \{1,\ldots,n\}$ and the edge set $E(G)$. For a vertex $x \in V(G)$, $N_G(x) = \{y \in V(G) \mid \{x,y\} \in E(G)\}$ denotes the set of neighbours of $x$. For a subset $U$ of vertices of $G$, $N_G(U) = \cup_{u \in U} N_G(u)$ denotes the set of all neighbours of $U$. The induced subgraph of $G$ on $U$, denoted by $G_U$, is a graph whose vertex set is $U$ and for any $u,v \in U$, $\{u,v\}$ is an edge of $G_U$ if and only if it is an edge of $G$. We refer to \cite{Di} for more information on graph theory.

The edge ideal of $G$ is defined to be
$$I(G)=(x_ix_j \mid \{i,j\}\in E(G))\subseteq S.$$

Let $\ww:E(G)\rightarrow \ZZ_{>0}$ be a weight function on edges of $G$. Paulsen and Sather-Wagstaff \cite{PS} introduced and studied the edge ideals of edge-weighted graphs of $G$
$$I(G_\ww) = \big( (x_i x_j)^{w(i,j)} \mid \{i,j\}\in E(G)\big) \subseteq S.$$
When $I(G_\ww)$ is Cohen-Macaulay, we call $(G,\ww)$ a Cohen-Macaulay edge-weighted graph. We have the following properties of associated radicals of edge ideals of edge-weighted graphs.
\begin{lem}\label{lem_associated_radicals_edge_weight} Let $G_{\ww}$ be a graph with a weight function $\ww: E(G) \to \NN$. For any $x^\a$, let $U=\{i \mid \text{there exists } j \text{ such that } \{i,j\}\in E(G) \text{ and }  a_i <  w(i,j) \le a_j\}$. Then  
$$\sqrt{I(G_{\ww}) : x^\a} = I(G\backslash U) + (x_i \mid i \in U),$$
where $I(G\backslash U)$ is the edge ideal of the induced subgraph of $G$ on $[n] \backslash U$.
\end{lem}
\begin{proof} Let $J = \sqrt{I(G_{\ww}) : x^\a}$. By \cite[Lemma 2.24]{MNPTV}, generators of $J$ are $x_ix_j$ with $x_i x_j \in I$ and $x_i$ for some variables $i$. Now $x_i \in J$ if and only if there exists an index $j$ such that 
$$x_i = \sqrt{(x_ix_j)^{w(i,j)}/\gcd((x_ix_j)^{w(i,j)} ,x^\a) }$$
In particular, we must have $a_i < w(i,j) \le a_j$. The conclusion follows.
\end{proof}

\begin{lem}\label{lem_associated_rad_1} Let $G_{\ww}$ be an edge-weighted graph. Assume that $J = \sqrt{I(G_{\ww}):x^\a}$ is an associated radical of $I(G_{\ww})$ and $K$ is an associated radical of $J$. Then $K$ is also an associated radical of $I(G_{\ww})$. 
\end{lem}
\begin{proof}
Since $J$ is radical, we may assume that $K = J : x^\b$. Let $\c$ be an exponent such that 
$$c_i = \begin{cases} \max \{ w(i,j) \mid j \in N(i) \} & \text{ if } i \in \supp \b\\
a_i &\text{ otherwise}.\end{cases}$$
We first claim that $x^{\c} \notin I(G_{\ww})$. By definition of $\c$, we need to prove that $x_i^{a_i} x_j^{c_j} \notin I(Y)$ for all $i \in \supp \a$ and all $j \in \supp \b$. Indeed, since $j \in \supp \b$, $x_j \notin J = \sqrt{I(G_\ww):x^\a}$. Hence, $a_i < w(i,j)$ for all $j \in N(i)$. Thus, $x_i^{a_i} x_j^{c_j} \notin I(G_{\ww})$. 

Let 
$$ U= \{i \mid \text{ there exists } j \text{ such that } a_i < w(i,j) \le a_j\}.$$
By Lemma \ref{lem_associated_radicals_edge_weight}, $J = I(G_{\ww}) + (x_i \mid i \in U)$. Thus, 
$$\sqrt{I(G_{\ww}) : x^\c} = I(G_{\ww}) + (x_i \mid i \in U \cup N(\supp \b)) = J : x^\b.$$
The conclusion follows.
\end{proof}

\begin{lem}\label{lem_associated_rad_2} Let $G_{\ww}$ be an edge-weighted graph. Assume that $\sqrt{I(G_{\ww}):x^\a} = I(G_{\ww}) + (x_i \mid i \in U)$ for some $U \subseteq [n]$. Let $I(G_{\ww}') = I(G_{\ww}) + (x_i \mid i \in U)$. Then an associated radical of $I(G_{\ww}')$ is also an associated radical of $I(G_{\ww})$. In particular, $\depth S/I(G_{\ww}) \le \depth S/I(G_{\ww}')$.    
\end{lem}
\begin{proof}Let $K = \sqrt{I(G_{\ww}'):x^\b}$ be an associated radical of $I(G_{\ww}')$. Let $\c = \max(\a,\b)$, i.e., $c_i = \max(a_i,b_i)$ for all $i = 1, \ldots, n$. We first prove that $x^{\c} \notin I(G_{\ww})$. By definition, it suffices to prove that $x_i^{a_i} x_j^{b_j} \notin I$ for all $i \in \supp \a$ and $j \in \supp \b$ such that $a_i \ge b_i$ and $a_j \le b_j$. Since $x_j^{b_j} \notin I(G_{\ww}')$, $a_i < w(i,j)$. Hence, $x^{a_i} x_j^{b_j} \notin I(Y)$. 

Let $V = \{i \mid \text{ there exists } j \text{ such that } w(i,j) \le b_j\}$. By Lemma \ref{lem_associated_radicals_edge_weight}, 
$$K = I(G_{\ww}') + (x_i \mid i \in V) = I(G_{\ww}) + (x_i \mid i \in U \cup V) =  \sqrt{I(G_{\ww}):x^{\c}}.$$
The last statement follows from the first statement and Theorem \ref{thm_Hochster_depth}.
\end{proof}

\subsection{Ferrers ideals and their depth and regularity}\label{subsection_Ferrers} 
Let $\lambda=(\lambda_1,\lambda_2,\ldots,\lambda_n)$ be a partition with $m = \lambda_1 \ge \lambda_2 \ge \ldots \ge \lambda_n \ge 1$. We denote by $\mu$ the conjugate partition. The Ferrers graph $G_\lambda$ is a bipartite graph on $X = \{x_1, \ldots, x_n\}$ and $Y = \{y_1,\ldots,y_m\}$ with $N(x_i) = \{y_1, \ldots, y_{\lambda_i}\}$ and $N(y_j) = \{x_1, \ldots, x_{\mu_j}\}$ for all $i = 1, \ldots, n$ and all $j = 1, \ldots, m$. The Ferrers ideal associated with the partition $\lambda$ is edge ideal of the Ferrers graph inside the polynomial ring $S = \k[x_1,\ldots,x_n,y_1,\ldots,y_m]$
$$I_{\lambda} = (x_iy_j \mid 1 \le i \le n, 1 \le j \le \lambda_i).$$
Corso and Nagel \cite{CN} computed all the Betti numbers of $I_\lambda$. In particular, they deduce the following results about the depth and regularity of $I_\lambda$.

\begin{prop} {\rm \cite[Corollary   2.2]{CN}} \label{CN_Ilambda}
	\begin{enumerate}
		\item $\height(I_{\lambda}) = \min\{\min_j \{\lambda_j + j - 1\}, n\}$,
		\item $\pd(S/I_{\lambda}) =  \max_j \{\lambda_j + j - 1\}$
		\item $\reg(S/I_{\lambda})=1$.
	\end{enumerate}  
\end{prop}

\begin{prop} {\rm \cite[Corollary 2.8]{CN}} \label{CM_Ilambda}
	The following conditions are equivalent:
	\begin{enumerate}
		\item  $I_{\lambda}$  is a Cohen-Macaulay ideal.
				\item $n=m$ and $\lambda=(n,n-1, \ldots, 2,1)$; 
	\end{enumerate}
	
\end{prop}

The depth and regularity of powers of $I_\lambda$ are also well-known.
\begin{prop}  For all $t \ge 2$, we have 
$$\depth S/I_{\lambda}^t=1.$$
\end{prop}
\begin{proof}It is easy to see that $I_\lambda^2 : (x_1y_1) = I(K_{X,Y})$, which is the edge ideal of a complete bipartite graph with bipartition $X \cup Y$. Hence, by Theorem \ref{thm_Hochster_depth}, $\depth S/I_\lambda^2 \le \depth S/I_\lambda^2 : (x_1 y_1) = 1$. The conclusion follows from \cite{T}.	 
\end{proof}

\begin{prop} For all $t\ge 2$, $I_{\lambda}^t$ has a linear free resolution. In particular, 
	 $\reg(S/I_{\lambda}^t) = 2t-1$. 
\end{prop}
\begin{proof} The Ferrers ideals is the edge ideal of a cochordal graph. The conclusion follows from \cite{HHZ}.  
\end{proof}

\subsection{Tableau ideals and their associated radicals} Let $Y$ be a filling of a Young diagram $\lambda$. Recall that the tableau ideal associated to $Y$ is 
$$I(Y) = ((x_iy_j)^{w(i,j)} \mid 1 \le i \le n, 1 \le j \le \lambda_i).$$
Let $\a \in \NN^n$ and $\b \in \NN^m$ be exponents, we denote by $x^\a y^\b$ the monomial 
$$x^\a y^\b = x_1^{a_1}\cdots x_n^{a_n} y_1^{b_1} \cdots y_m^{b_m} \in S.$$ 
We have the following properties of associated radicals of tableau ideals.

\begin{lem}\label{lem_associated_rad_3}
    Let $\lambda$ be a partition. There exists an associated radical $J$ of $I_\lambda$ such that $\depth S/I_\lambda = \depth S/J$ and $J$ corresponds to a partition whose all rows have the same lengths (i.e., a complete bipartite graph).
\end{lem}
\begin{proof} Let $\sigma$ be the smallest index such that $\lambda_\sigma + \sigma = \max (\lambda_i + i \mid i = 1, \ldots, n)$. Let $\alpha$ be the largest index such that $\lambda_\alpha > \lambda_\sigma$. Since $\lambda_\sigma +\sigma \ge \lambda_{\sigma +1} + \sigma + 1$, we deduce that $\lambda_{\sigma+1} < \lambda_{\sigma}$. Let $J = I_\lambda : (y_{\lambda_\alpha} x_{\sigma+1})$. Then 
\begin{align*}
    J & = I_\lambda + (x_1,\ldots, x_\alpha) + (y_1, \ldots, y_{\lambda_{\sigma+1}}) \\
    &= (x_{\alpha+1},\ldots, x_\sigma)(y_{\lambda_{\sigma+1}+1},\ldots,y_{\lambda_\sigma}) + (x_1,\ldots, x_\alpha) + (y_1, \ldots, y_{\lambda_{\sigma+1}}).
\end{align*}
In particular, $U = \{x_i \mid i \ge \sigma + 1\}$ and $V = \{y_j \mid j \ge \lambda_\sigma + 1\}$ are the sets of free variables of $J$. By Lemma \ref{lem_depth_variables_clearing}, $\depth S/J = |U| + |V| + 1$. The conclusion follows from Proposition \ref{CN_Ilambda}.    
\end{proof}

\begin{exam}  Given a Young diagram as follows: 
\begin{center}
$\begin{ytableau}
*(white)  & *(white)  & *(white)  & *(white)  & *(white)  & *(white)  & *(white) \\
*(white)  & *(white)  & *(white)  & *(white)  & *(white)   & *(white)  & *(yellow) \\
*(white)  & *(white)  & *(white)  & *(white)  & *(white)  & *(blue)   \\
*(white)  & *(white)  & *(white)  & *(white)  & *(white)  & *(blue) \\
*(white)  & *(white)  & *(white)  & *(white)  & *(yellow) \\
*(white)  & *(white) & *(white) 
\end{ytableau}$
 \end{center}
We take the colon with respect to the variables corresponding to the yellow boxes, namely, $y_7$ and $x_5$. Then we have $J = (x_3,x_4)y_6 + (x_1,x_2,y_1,\ldots, y_5)$. Then $x_5,x_6$ and $y_7$ will be free variables and $J$ corresponds to the Ferrers ideal with just two rows of length $1$, corresponding to the blue boxes.
\end{exam}

Thus, we have 
\begin{lem}\label{depth_complete_bipartite}
    Let $I(Y)$ be a tableau ideal. There exists an associated radical $J$ of $I(Y)$ such that $\depth S/I(Y) = \depth S/J$ and $J$ corresponds to a partition whose all rows have the same lengths.
\end{lem}
\begin{proof}
    Follows from Theorem \ref{thm_Hochster_depth}, Lemma \ref{lem_associated_rad_1} and Lemma \ref{lem_associated_rad_3}.
\end{proof}

\section{Depth of tableau ideals}\label{sec_depth_tableaux_ideal}

In this section, we compute the depth of tableau ideals. We fix the following notation throughout the section:
\begin{enumerate}
    \item $\lambda = (\lambda_1,\ldots, \lambda_n)$ is a partition with $\lambda_1 = m$.
    \item $\alpha = \alpha(\lambda)$ is the smallest index $i$ such that $\lambda_i + i = \max (\lambda_j + j \mid j = 1, \ldots, n).$
    \item $Y$ is a filling of $\lambda$ with the value $w_{ij}$ on the $ij$-th box of $\lambda$. 
    \item $\omega = \min (w(i,j)  \mid 1 \le i \le n, 1 \le j \le \lambda_i)$ is the minimum weight.
\end{enumerate}

\begin{defn}
    A box $(\gamma,\delta)$ is called a {\it minimal box} of $Y$ if $w(\gamma,\delta) = \omega$ and $w(i,j) > \omega$ for all $(i,j) \neq (\gamma,\delta)$ such that $i \le \gamma$ and $j \le \delta$.
\end{defn}
We then fix $(\gamma,\delta)$ a minimal box of $Y$. We first see the change in the depth of Ferrers ideals after adding a variable.

\begin{lem}\label{lem_depth_adding_a_variable} Let $J = I_\lambda + (x_a)$. Then 
$$\depth S/J = \begin{cases}
    \depth S/I_\lambda - 1 & \text{ if } a \ge \alpha + 1,\\
    \depth S/I_\lambda & \text{ if } a \le \alpha - 1,\\
    \depth S/I_\lambda + \epsilon & \text{ if } a = \alpha,
\end{cases}$$
where $\epsilon$ is a non-negative integer. 
\end{lem}
\begin{proof}
Denote by $\mu$ the partition obtained by deleting the $a$th row of $\lambda$, i.e., $\mu=(\lambda_1,\ldots, \lambda_{a-1}, \lambda_{a+1},\ldots, \lambda_n)$. Then $J$ is isomorphic to $I_\mu + (x_a)$. Let $R$ be the polynomial subring on $\{x_1,\ldots, x_{a-1},x_{a+1},\ldots,x_n,y_1,\ldots,y_m\}$. By Proposition \ref{CN_Ilambda}, 
\begin{align*}
    \pd (R/I_{\mu}) & = \max_{i} \{ \mu_i + i - 1\} \\
    &= \max ( \max \{\lambda_i + i - 1 \mid i \le a-1\}, \max \{\lambda_i + i-2 \mid i \ge a+1\} ).
\end{align*}

We next distinguish some cases:

\vskip0.5em \noindent {\bf Case 1.} $a \ge \alpha + 1$. In this case, we have $\pd (R/I_\mu) = \pd (S/I_\lambda)$. By the Auslander-Buchsbaum formula, we deduce that $\depth S/J = \depth S/I - 1$. 

\vskip0.5em \noindent {\bf  Case 2.} $a \le \alpha - 1$. In this case, we have $\pd (R/I_\mu) = \pd (S/I_\lambda) - 1$. In particular, we have $\depth S/J = \depth S/I$. 

\vskip0.5em \noindent {\bf  Case 3.} $a = \alpha$. Since $\alpha$ is the smallest index such that $\lambda_i + i = \max (\lambda_j + j)$, we have $\lambda_i + i < \lambda_\alpha + \alpha$ for all $i < \alpha$. In particular, we deduce that $\pd(R/I_\mu) \le \pd (S/I_\lambda) - 1$. Hence, $\depth S/J = \depth S/I + \epsilon$ for a non-negative integer $\epsilon$.
\end{proof}

We now have the following two key lemmas to compute the depth of tableau ideals. 
\begin{lem}\label{lem_depth_reduction_w_1} Assume that $Y$ is a filling of a Young diagram $\lambda$. Let $(\gamma,\delta)$ be a minimal box of $Y$. Let $Y'$ be a filling of $\lambda$ obtained by replacing the value $\omega$ on the $(\gamma,\delta)$ box by $1$. Then
    $$\depth S/I(Y) \ge \depth S/I(Y').$$
\end{lem}
\begin{proof}
    By Theorem \ref{thm_Hochster_depth}, there exists an exponent $x^\a y^\b$ such that $\depth S/I(Y) = \depth S/\sqrt{I(Y):x^\a y^\b}$. Since $x^\a y^\b \notin I(Y)$, we must have either $a_\gamma$ or $b_\delta < \omega$. We may assume that $a_\gamma < \omega$. Let $\a'$ be an exponent such that $$a'_i = \begin{cases} a_i & \text{ if } i \neq \gamma\\
    0 & \text{ if } i  = \gamma. \end{cases}$$ 
    There are two cases: 
   
   \vskip0.5em \noindent {\bf Case 1.} $b_\delta < \omega$. Let $\b'$ be an exponent such that 
   $$b'_j = \begin{cases} b_j & \text{ if } j \neq \delta\\
    0 & \text{ if } j  = \delta. \end{cases}$$
    By Lemma \ref{lem_associated_radicals_edge_weight}, 
    $$\sqrt{I(Y'):x^{\a'} y^{\b'}} = \sqrt{I(Y):x^\a y^\b}.$$

    \vskip0.5em \noindent {\bf Case 2.} $b_\delta \ge \omega$. By Lemma \ref{lem_associated_radicals_edge_weight},
    $$\sqrt{I(Y'):x^{\a'} y^\b} = \sqrt{I(Y):x^\a y^\b}.$$
    In either cases, from Theorem \ref{thm_Hochster_depth} we deduce that $\depth S/I(Y) \ge \depth S/I(Y')$.    
\end{proof}

\begin{lem}\label{lem_depth_upperbound_add_variable}
    Let $(\gamma,\delta)$ be a minimal box of $Y$. Then 
    $$\depth S/I(Y) \le \depth S/(I(Y) + (x_\gamma)).$$
\end{lem}
\begin{proof} By Lemma \ref{depth_complete_bipartite}, there exists an exponent $x^\a y^\b$ such that
$$J = \sqrt{I(Y) + (x_\gamma) : x^\a y^\b} = (x_j \mid j \in U_1) (y_k \mid k \in V_1) + (x_\gamma, x_j , y_k \mid j \in U_2, k \in V_2)$$
for some subsets $U_1, U_2 \subseteq [n]$ and $V_1,V_2 \subseteq [m]$ and $\depth S/( I(Y) + (x_\gamma)) = \depth S/J.$ Since $x^\a y^\b \notin (I(Y) + (x_\gamma))$, $a_\gamma = 0$. For ease of reading, we divide the proof into several steps.

\vskip0.5em \noindent {\bf Step 1.} Reduction to the case $x_\gamma \notin \sqrt{I(Y):x^\a y^\b}$. Assume that $x_\gamma \in \sqrt{I(Y):x^\a y^\b}$. By Lemma \ref{lem_associated_radicals_edge_weight},
$\sqrt{I(Y) : x^\a y^\b} = J$. By Theorem \ref{thm_Hochster_depth}, $\depth S/I(Y) \le \depth S/J = \depth S/(I(Y) + (x_\gamma))$. Thus, we may assume that $x_\gamma \notin \sqrt{I(Y) :x^\a y^\b}$. 

\vskip0.5em \noindent {\bf Step 2.} Reduction to the case $b_j = 0$ for all $j = 1, \ldots, m$. If $U_2 = \emptyset$ then we may replace $\b$ by $\bf{0}$ without changing the radical ideal and obtain the reduction step. Thus, we assume that $b_j \neq 0$ for some $j$ and $U_2$ is non-empty. By Step 1, $\gamma \notin U_2$. Let $I(Z) = I(Y) + (x_i \mid i \in U_2)$. By Lemma \ref{lem_associated_radicals_edge_weight}, $J = \sqrt{(I(Z) + (x_\gamma)):x^\a}$. By Theorem \ref{thm_Hochster_depth}, $\depth S/(I(Z) + (x_\gamma)) \le \depth S/J = \depth S/(I(Y) + (x_\gamma))$. By Lemma \ref{lem_associated_rad_2}, $\depth S/(I(Z) + (x_\gamma)) \ge \depth S/ (I(Y) + (x_\gamma))$. By induction, we deduce that 
$$\depth S/I(Z) \le \depth S/I(Z) + (x_\gamma)) = \depth S/(I(Y) + (x_\gamma)).$$
By Lemma \ref{lem_associated_rad_2}, $\depth S/I(Y) \le \depth S/I(Z)$. Thus, we may assume that $b_j = 0$ for all $j$. In particular, we have 
\begin{equation}\label{eq_depth_2}
    J = \sqrt{(I(Y) + (x_\gamma)):x^\a} =  (x_j \mid j \in U_1)(y_k \mid k \in V_1) + (x_\gamma,y_k \mid k \in V_2).
\end{equation}

\vskip0.5em \noindent {\bf Step 3.} Reduction to the case $a_i = 0$ for all $i = 1, \ldots, n$. Similar to Step 2, assume that $a_i \neq 0$ for some $i$ and $V_2$ is non-empty. Let $I(Z) = I(Y) + (y_k \mid k \in V_2)$. By definition, we have $J = \sqrt{I(Z) + (x_\gamma)}$. By Theorem \ref{thm_Hochster_depth} and Lemma \ref{lem_associated_rad_2}, $\depth S/(I(Z) + (x_\gamma)) = \depth S/(I(Y) + (x_\gamma)$. By induction and Lemma \ref{lem_associated_rad_2}, we deduce that 
$$\depth S/I(Y) \le \depth S/I(Z) \le \depth S/(I(Z) + (x_\gamma)) = \depth S/(I(Y) + (x_\gamma)).$$
Thus, we may assume that $a_i = 0$ for all $i = 1, \ldots, n$. In particular, we have 
\begin{equation}\label{eq_depth_3}
J = \sqrt{I(Y) + (x_\gamma)} = I_\lambda + (x_\gamma) = (x_j \mid j \in U_1) (y_k \mid k \in V_1) + (x_\gamma).
\end{equation}

\vskip0.5em \noindent {\bf Step 4.} Reduction to the case $x_\gamma$ is the last variable, i.e., $\gamma = n$. By Eq. \eqref{eq_depth_3}, we have $U_1 \cup \{\gamma\} = [n]$. Let $V_2 = [m] \setminus V_1$. By Lemma \ref{lem_depth_variables_clearing}, $\depth S/J = |V_2| + 1$. Assume that $\gamma < n$. If $\gamma > 1$ then from Eq. \eqref{eq_depth_3}, we deduce that $\lambda_1 = \cdots = \lambda_n$ and $V_2 = \emptyset$. In this case, we have $\depth S/I(Y) \le \depth S/I_\lambda = 1 = \depth S/J$. Now, assume that $\gamma = 1$. By Lemma \ref{lem_depth_adding_a_variable}, 
$$\depth S/I_\lambda \le \depth S/(I_\lambda + (x_1)) = \depth S/J.$$
The conclusion follows from Theorem \ref{thm_Hochster_depth}. Thus, we may assume that $x_\gamma$ is the last variable. 

\vskip0.5em \noindent {\bf Step 5.} Conclusion step. By Step 3 and Step 4, $\gamma = n$ and $\lambda_1 = \ldots = \lambda_{n-1}$. Since $(\gamma,\delta)$ is a minimal box, $w(j,\delta) > \omega$ for all $j < \gamma$. By Lemma \ref{lem_associated_radicals_edge_weight}, 
$$\sqrt{I(Y) : y_\delta^\omega} = I_\lambda + (x_\gamma) = J.$$
The conclusion follows from Theorem \ref{thm_Hochster_depth}.
\end{proof}

We are now ready for the main result of the section.

\begin{thm}\label{thm_depth_tableau}
    Let $Y$ be an arbitrary filling of a Young diagram. Let $(\gamma,\delta)$ be a minimal box of $Y$. Then 
    $$\depth S/I(Y) = \min \{ \depth S/(I(Y) + (x_\gamma)), \depth S/(I(Y) + (y_\delta)) \}.$$
\end{thm}
\begin{proof}
    By Lemma \ref{lem_depth_reduction_w_1} and Lemma \ref{lem_depth_upperbound_add_variable}, we may assume that $\omega = 1$ and we need to prove that 
    $$\depth S/I(Y) \ge \min \{ \depth S/(I(Y)+(x_\gamma)), \depth S/(I(Y)+(y_\delta)) \}.$$
    By Lemma \ref{lem_depth_reg_colon}, it suffices to prove that $\depth S/I(Y) : x_\gamma \ge \depth S/(I(Y) + (y_\delta))$. By Theorem \ref{thm_Hochster_depth}, there exists an exponent $x^\a y^\b$ such that $\depth S/I(Y) : x_\gamma = \depth S/\sqrt{I(Y):x_\gamma x^\a y^\b}$. Since $x^\a y^\b \notin I(Y) :x_\gamma$, $b_\delta = 0$. In other words, $x_\gamma x^\a y^\b \notin I(Y) + (y_\delta)$. By Lemma \ref{lem_associated_radicals_edge_weight}, 
    $$\sqrt{(I(Y) + (y_\delta)):x_\gamma x^\a y^\b} = \sqrt{I(Y) : x_\gamma x^\a y^\b}.$$ The conclusion follows from Theorem \ref{thm_Hochster_depth}.
\end{proof}

 As a corollary, we have 
 \begin{cor}\label{cor_depth_weakly_increasing} Assume that $Y$ is a weakly increasing filling of a Young diagram $\lambda$. Then $$\depth S/I(Y) = \depth S/I_\lambda = m+n - \max_i \{ \lambda_i + i -1\}.$$ 
 \end{cor}
\begin{proof} Since $Y$ is weakly increasing, $(1,1)$ is a minimal box of $Y$. By Theorem \ref{thm_depth_tableau}, we have 
$$\depth S/I(Y) = \min \{ \depth (S/I(Y) + (x_1)), \depth S/(I(Y) + (y_1)) \}.$$
By Theorem \ref{thm_depth_tableau}, we also have 
$$\depth S/I_\lambda = \min \{ \depth (S/I_\lambda + (x_1)), \depth S/(I_\lambda + (y_1)) \}.$$
The conclusion follows from induction and Proposition \ref{CN_Ilambda}.
\end{proof}

\begin{cor}
     Let $Y$ be an arbitrary filling of a Young diagram $\lambda$. Then $I(Y)$ is Cohen-Macaulay if and only if $\lambda = (n,n-1,\ldots,2,1)$ and $Y$ is weakly increasing.
 \end{cor}
 \begin{proof} First, assume that $I(Y)$ is Cohen-Macaulay. Then $I_\lambda$ is Cohen-Macaulay. By Proposition \ref{CM_Ilambda}, $\lambda = (n,n-1,\ldots,2,1)$. Assume by contradiction that $Y$ is not weakly increasing. Let $j$ be the smallest index such that $w(i,j) > w(i,j+1)$ for some $i$. We have 
 $$\sqrt{I(Y) : x_i^{w(i,j+1)}} = I_\lambda + (y_{j+1}, y_\ell \mid \ell \in V)$$
 for some subset $V$ of $\{j+2, \ldots, n\}$. By Proposition \ref{CN_Ilambda} and Theorem \ref{thm_Hochster_depth}, 
 $$\depth S/I(Y) < n,$$
 a contradiction.

     Now assume that $\lambda = (n,n-1,\ldots,1)$ and $Y$ is weakly increasing. By Corollary \ref{cor_depth_weakly_increasing}, $\depth S/I(Y) = \depth S/I_\lambda$. The conclusion follows.
 \end{proof}

 \begin{rem} By \cite{FSTY}, we know that $I(Y)$ is Cohen-Macaulay if and only if it is unmixed. However, it was not known that this is equivalent to the condition of weakly increasing weights.  
 \end{rem}
 
\begin{exam}
    Consider the following filling of a Young diagram 
    \begin{center}
\begin{tabular}{ccc}
$\begin{ytableau}
*(white) 1 & *(white) 3 & *(white) 5 & *(white) 6 & *(white) 7\\
*(white) 2 & *(white) 3 & *(yellow) 4 & *(white) 6   \\
*(white) 2 & *(white) 3 & *(white) 5   \\
*(white) 2 & *(white) 4\\
*(white) 3
\end{ytableau}$\\[5pt]
\end{tabular}
 \end{center}
Let $I = I(Y)$ and $\lambda$ be the corresponding partition. Then we have $\depth S/I(Y) = 4 = \depth S/I_\lambda-1$ and $\depth S/I(Y)^2 = 2$. That $\depth S/I(Y) = 4$ is because $Y$ is not weakly increasing, and there is only one place (the yellow box) where it is non weakly increasing. Thus, in general, $\depth S/I(Y)^t$ could be larger than $1$ for $t \ge 2$, in contrast to the depth of powers of Ferrers ideals.
\end{exam}

\section{Regularity of tableau  ideals}\label{sec_regularity_tableaux_ideal}

In this section, we compute the regularity of tableau ideals. First, we have the following remark about the homology of the degree complexes of tableau ideals.

\begin{lem}\label{lem_homology} Let $Y$ be an arbitrary filling of a Young diagram $\lambda$. Let $(x^\a y^\b, i)$ be a critical pair of $I(Y)$. Then $i = 0$ or $i = 1$.
\end{lem}
\begin{proof}
By Lemma \ref{lem_associated_radicals_edge_weight}, $J = \sqrt{I(Y):x^\a y^\b} = I_\lambda + (x_i \mid i \in U) + (y_j \mid j \in V)$ for some subsets $U$ and $V$ of $[n]$ and $[m]$ respectively. The conclusion follows from Lemma \ref{Key0} and the fact that $\reg (S/J) \le 1$.
\end{proof}

We now have a simple formula when $\lambda$ has only one row (or one column).

\begin{lem}\label{lem_reg_one_row} Assume that $\lambda = (\lambda_1)$ has only one row. Then 
$$\reg S/I(Y) = \sum_{j=1}^m (w(1,j) - 1) + \max\{ w(1,j) \mid j = 1, \ldots, m\}.$$    
\end{lem}
\begin{proof} Let $(x^\a y^\b,i)$ be an extremal pair of $I(Y)$. By Remark \ref{rem_critical_sets}, we must have $b_j \le w(1,j) - 1$ and $a_1 \le \max \{w(1,j) \mid j = 1, \ldots, m\} - 1$. Furthermore, by Lemma \ref{lem_homology}, we deduce that $i \le 1$. Hence, 
$$\reg S/I(Y) \le \sum_{j=1}^m (w(1,j) - 1) + \max\{ w(1,j) \mid j = 1, \ldots, m\}.$$
Now, we choose $a_1 = \max \{w(1,j) \mid j = 1, \ldots, m\} - 1$ and $b_j = w(1,j) - 1$. Then
$$\sqrt{I(Y) : x^\a y^\b} = I_\lambda + (y_j \mid w(1,j) \le  a_1).$$
In particular, $\h_0(\Delta_{(\a,\b)}(I(Y));\k) \neq 0$. The conclusion follows from Lemma \ref{Key0}.    
\end{proof}

We may now assume that $Y$ has at least two rows and two columns. As in the previous section, we fix the following notation throughout the section:
\begin{enumerate}
    \item $\lambda = (\lambda_1,\ldots, \lambda_n)$ is a partition with $\lambda_1 = m$.
    \item $Y$ is a filling of $\lambda$ with the value $w_{ij}$ on the $ij$-th box of $\lambda$. 
    \item $\omega = \min (w(i,j)  \mid 1 \le i \le n, 1 \le j \le \lambda_i)$ is the minimum weight.
    \item $(\gamma,\delta)$ is a minimal box of $Y$.
\end{enumerate}

We now give a structure result about Stanley-Reisner complexes of bipartite graphs that will be useful later.
\begin{lem}\label{lem_homology_0_bipartite_graph} Let $\Delta = \Delta(I(G))$ be the simplicial complex of the edge ideal of a non-empty bipartite graph $G$. Assume that $\h_0(\Delta;\k) \neq 0$. Then $G$ is a complete bipartite graph. In other words, $\Delta = \langle X, Y\rangle$ with $X,Y$ are the partitions of $V(G)$.    
\end{lem}
\begin{proof}
    Since $\h_0(\Delta;\k) \neq 0$, $\depth S/I(G) \le 1$. Since $G$ is bipartite, we deduce that $\depth S/I(G) = 1$. The conclusion follows from \cite[Lemma 3.1]{T}.
\end{proof}
We also have the following two key lemmas to compute the regularity of $I(Y)$. 

\begin{lem}\label{lem_reg_reduction_to_weight_1} Assume that $Y$ has more than one box. Let $(\gamma,\delta)$ be a minimal box of $Y$. Let $Y'$ be a filling obtained by replacing the value $\omega$ on the $(\gamma,\delta)$ box by $1$. Then 
$$\reg S/I(Y) \le \reg S/I(Y') + (\omega - 1).$$   
\end{lem}
\begin{proof} If $\omega = 1$ the conclusion is vacuous. Thus, we may assume that $\omega > 1$. Let $(x^\a y^\b,i)$ be an extremal pair of $S/I(Y)$, i.e., we have $\reg S/I(Y) = |\a| + |\b| + i$ and there exists a face $F$ of $\Delta = \Delta_{(\a,\b)}(I(Y))$ such that $\h_{i-1}(\lk_{\Delta} F;\k) \neq 0$. For ease of reading we divide the proof into several steps.

  \vskip0.5em \noindent {\bf Step 1.} Reduction to the case $a_\gamma < \omega$ and $b_\delta < \omega$. Indeed, assume that $a_\gamma \ge \omega$. Then $b_\delta < \omega$. Let $\b'$ be an exponent such that 
  \begin{equation}\label{eq_b}
      b'_j = \begin{cases} b_j & \text{ if } j \neq \delta, \\
  0 & \text{ if } j = \delta.\end{cases}
  \end{equation}
  By Lemma \ref{lem_associated_radicals_edge_weight}, 
  $$\sqrt{I(Y) : x^\a y^\b} = \sqrt{I(Y'):x^\a y^{\b'}}.$$
  By Lemma \ref{Key0}, $(x^\a y^{\b'},i)$ is a critical exponent of $I(Y')$. Hence, 
  $$\reg S/I(Y') \ge |\a| + |\b'| + i \ge \reg S/I(Y) - (\omega - 1).$$
Thus, we may assume that $a_\gamma < \omega$ and $b_\delta < \omega$. 

\vskip0.5em \noindent {\bf Step 2.} Reduction to the case $a_\gamma > 0$ and $b_\delta > 0$. Assume that $a_\gamma = 0$, we set $\b'$ as in Eq. \eqref{eq_b}. Then,
  $$\sqrt{I(Y) : x^\a y^\b} = \sqrt{I(Y'):x^\a y^{\b'}},$$
  and we deduce the required inequality by Lemma \ref{Key0}. Thus, we may assume that $0 < a_\gamma < \omega$ and $0 < b_\delta < \omega$. 
  
 \vskip0.5em \noindent {\bf Step 3.} Reduction to the case $x_\gamma, y_\delta \notin \sqrt{I(Y):x^\a y^\b}$. Let $\a'$ be an exponent such that   
  \begin{equation}\label{eq_a}
      a'_j = \begin{cases} a_j & \text{ if } j \neq \gamma, \\
  0 & \text{ if } j = \gamma.\end{cases}
  \end{equation}
If $x_\gamma \in \sqrt{I(Y):x^\a y^\b}$ then $$\sqrt{I(Y'):x^{\a'}y^{\b}} =       \sqrt{I(Y) : x^{\a} y^{\b}}.$$
In other words $(x^{\a'} y^{\b},i)$ is a critical pair of $I(Y')$. By Step 1, $|\a'| \ge |\a| - (\omega-1)$. The conclusion follows from Lemma \ref{Key0}. Thus, we obtain the reduction step.

\vskip0.5em \noindent {\bf Step 4.} Reduction to the case $F = \emptyset$. By Step 2, $a_\gamma, b_\delta > 0$. Thus, $x_\gamma,y_\delta \notin \supp F$. Assume that $F \neq \emptyset$, say $z \in F$. By Lemma \ref{lem_regularity_equality_adding_variable}, $\reg S/I(Y) = \reg S/(I(Y) + (z))$. By induction, we have 
$$\reg S/(I(Y) + (z)) \le \reg S/(I(Y') + (z)) + (\omega - 1).$$
By Lemma \ref{lem_depth_reg_adding_variable}, we deduce that 
$$\reg S/I(Y) \le \reg S/I(Y') + (\omega- 1),$$
as required. Thus, we may assume that $F = \emptyset.$

\vskip0.5em \noindent {\bf Step 5.} Reduction to the case $i = 1$. Let 
$$ \Delta = \Delta_{(\a,\b)} (I(Y)), \Gamma = \Delta_{(\a',\b)} (I(Y')), \text{ and }    \Sigma = \Delta_{(\a,\b')} (I(Y')).$$
We have $\Gamma$ and $\Sigma$ are the restriction of $\Delta$ to $X \cup Y \setminus \{x_\gamma\}$ and $X \cup Y \setminus \{y_\delta\}$ respectively. By Lemma \ref{lem_homology}, $i = 0$ or $i = 1$. Assume that $i = 0$. From the exact sequence
$$\cdots \to \h_{-1} (\lk_\Delta x_\gamma;\k) \to \h_{-1} (\Gamma;\k) \to \h_{-1} (\Delta;\k) \to 0$$
we deduce that $\h_{-1}(\Gamma;\k) \neq 0$. Thus, $(x^{\a'}y^{\b},i)$ is a critical pair of $I(Y')$. By Lemma \ref{Key0}, the conclusion follows. Thus, we may assume that $i = 1$.

\vskip0.5em \noindent {\bf Step 6.} Reduction to the case $\Delta = \langle \{x_\gamma\}, \{y_\delta\} \rangle$. By Step 5, we assume that $i = 1$, i.e., $\h_{0}(\Delta;\k) \neq 0$. Consider the long exact sequence
$$\cdots \to \h_0(\Gamma;\k) \to \h_0(\Delta;\k) \to \h_{-1} (\lk_\Delta x_\gamma; \k) \to  \cdots \to 0.$$
If $\h_0(\Gamma;\k) \neq 0$, the conclusion also follows from Lemma \ref{Key0}. Thus, we may assume that $\h_0(\Gamma;\k) = 0$. Hence, $\h_{-1} (\lk_{\Delta} x_\gamma;\k) \neq 0$. Similarly, we may assume that $\h_{-1}(\lk_\Delta y_\delta;\k) \neq 0$. Thus, we have $\Delta = \{x_\gamma\} \cup \{y_\delta\} \cup \Delta'$ where $\Delta'$ is a simplicial complex on $X \cup Y \setminus \{x_\gamma,y_\delta\}$. If $x_j \in \Delta'$ then $x_\gamma x_j \in I_\Delta$, which is a contradiction. Thus, we must have $\Delta'$ is the empty simplicial complex. In other words, $\Delta = \langle \{x_\gamma\},\{y_\delta\}\rangle$ and $I_\Delta = \sqrt{I(Y):x^\a y^\b} = (x_\gamma y_\delta) + (x_j \mid j \neq \gamma) + (y_\ell \mid \ell \neq \delta)$.

\vskip0.5em \noindent {\bf Step 7.} Conclusion step. Let $U = \{j \mid j \neq \gamma \text{ and } a_j \neq 0\}$ and $V = \{ j \mid j \neq \delta \text{ and } b_j \neq 0\}$. Let $D$ be a directed bipartite graph on vertices $U \cup V$ whose edges are $(j,\ell)$ if $a_j < w(j,\ell) \le b_\ell$ and $(\ell,k)$ if $b_\ell < w(k,\ell) \le a_k$. If $a_j = 0$ and $b_\ell = 0$ for all $j \neq \gamma$ and all $\ell \neq \delta$ then $\reg S/I(Y) = 2\omega - 1$. Since $Y$ has at least two boxes, by degree reason, we also have $\reg S/I(Y') \ge 2\omega - 1$. Thus, the conclusion is clear. Hence, we may assume that $|U| + |V| \ge 1$. By Step 6, for any $j \in U$, $x_j \in \sqrt{x^\a y^\b}$. Thus, there must exist $\ell \in V$ such that $a_j < w(j,\ell) \le b_\ell$. In other words, in $D$, there is an edge going out from $j$. Similarly, if $\ell \in V$, there must exist an edge going out from $\ell$. Thus $U$ and $V$ are non-empty sets and $D$ has a directed cycle (the directed cycle might have length $2$). Assume that $u_1, v_1,u_2,v_2,\ldots, u_s,v_s$ is a directed cycle of $D$ where $u_j \in U$ and $v_j\in V$. By definition, $a_{u_j} < w(u_j,v_j) \le b_{v_j}$ and $b_{v_j} < w(u_{j+1},v_j) \le a_{u_{j+1}}$ for any $j = 1, \ldots, s$ where $u_{s+1} = u_1$. This is a contradiction. That concludes the proof of the Lemma.
\end{proof}

\begin{lem}\label{lem_reg_lower_bound} Let $Y$ be a filling of a Young diagram with at least two boxes. Let $(\gamma,\delta)$ be a minimal box of $Y$ with $w(\gamma,\delta) = \omega$. We have 
$$\reg S/I(Y) \ge (\omega -1) + \reg (S/(I(Y) + (x_\gamma)).$$    
\end{lem}
\begin{proof} If $\omega =1$ the conclusion follows from Lemma \ref{lem_depth_reg_adding_variable}. Thus, we may assume that $\omega > 1$. Let $(x^\a y^\b,i)$ be an extremal pair of $J = I(Y) + (x_\gamma)$. In particular, $a_\gamma = 0$. Let $\Gamma = \Delta_{(\a,\b)}(J)$ and $F$ is a face of $\Gamma$ such that $F \cap (\supp \a \cup \supp \b)  = \emptyset$ and $\h_{i-1}(\lk_\Gamma F;\k) \neq 0$. For ease of reading, we divide the proof into several steps. 

\vskip0.5em \noindent {\bf Step 1.} Reduction to the case $b_j \le w(\gamma,j) - 1$ for all $j = 1, \ldots,\lambda_\gamma$. Let $\a'$ be an exponent such that $a'_\gamma = \omega-1$ and $a'_i = a_i$ otherwise. If $\sqrt{J:x^\a y^\b} = \sqrt{I(Y):x^{\a'} y^\b}$ then $(x^{\a'}y^\b,i)$ is a critical pair of $I(Y)$. By Lemma \ref{Key0}, the conclusion follows. By Lemma \ref{lem_associated_radicals_edge_weight}, we may assume that $x_\gamma \notin \sqrt{I(Y):x^{\a'} y^{\b}}$ and 
\begin{equation}
    \sqrt{J:x^{\a}y^{\b}} = \sqrt{I(Y):x^{\a'}y^{\b}}  + (x_\gamma).
\end{equation}
Hence, $b_j \le w(\gamma,j) - 1$ for all $j = 1, \ldots, \lambda_\gamma$.

\vskip0.5em \noindent {\bf Step 2.} Reduction to the case $F = \emptyset.$ Assume that $F \neq \emptyset$ and $z$ is an element of $F$. By Lemma \ref{lem_regularity_equality_adding_variable}, $\reg S/J = \reg S/(J,z)$. Since $z \notin \{x_\gamma,y_\delta\}$, by induction, 
$$\reg S/(I(Y),z) \ge (\omega - 1) + \reg (S/(J,z)).$$
By Lemma \ref{lem_depth_reg_adding_variable}, we deduce that 
$$\reg S/I(Y) \ge \reg S/(I(Y),z) \ge (\omega - 1) + \reg S/J.$$
Thus, we may assume that $F = \emptyset$. 

\vskip0.5em \noindent {\bf Step 3.} Reduction to the case $b_j > 0$ for all $j = 1, \ldots, \lambda_\gamma$. Assume that $b_j = 0$ for some $j < \lambda_\gamma$. Let $\b'$ be an exponent such that 
 \begin{equation}\label{eq_3}
      b'_\ell = \begin{cases} b_\ell & \text{ if } \ell \neq j, \\
  \omega - 1 & \text{ if } \ell = j.\end{cases}
  \end{equation}
By Lemma \ref{lem_associated_radicals_edge_weight}, 
$$\sqrt{J:x^\a y^\b} = \sqrt{J: x^\a y^{\b'}}.$$
By Step 2, $F = \emptyset$. Thus $(x^{\a}y^{\b'},i)$ is also a critical pair of $J$. Since $|\a| + |\b'| > |\a| + |\b|$, this is a contradiction. Thus, we may assume that $b_j > 0$ for all $j = 1, \ldots, \lambda_\gamma$. 

\vskip0.5em \noindent {\bf Step 4.} Reduction to the case $w(\gamma,j) = \omega$ for all $j = 1, \ldots, \lambda_\gamma$. Let $Z$ be the tableau obtained by replacing all weights on the $\gamma$th row of $Y$ by $\omega$. Then $I(Y) + (x_\gamma) = I(Z) + (x_\gamma)$. We will now claim that $\reg S/I(Y) \ge \reg S/I(Z)$. Let $(\a,\b)$ be an extremal exponent of $I(Z)$. By Remark \ref{rem_critical_sets}, we have $a_\gamma \le \omega - 1$. By Lemma \ref{lem_associated_radicals_edge_weight},
$$\sqrt{I(Z) : x^\a y^\b} = \sqrt{I(Y) : x^\a y^\b},$$
By Lemma \ref{Key0}, $\reg S/I(Y) \ge \reg S/I(Z)$. Thus, we may replace $Y$ by $Z$ and obtain the reduction step.

\vskip0.5em \noindent {\bf Step 5.} Reduction to the case $i = 1$. Assume that $i = 0$, i.e., $\h_{-1}(\Gamma;\k) \neq 0$. Let $\Delta = \Delta_{(\a',\b)}(I(Y))$. Then $\Gamma$ is the restriction of $\Delta$ to $X \cup Y \setminus \{x_\gamma\}$. Consider the long exact sequence 
$$\cdots \to \h_{-1} (\lk_\Delta x_\gamma; \k) \to \h_{-1} (\Gamma ; \k) \to \h_{-1}(\Delta;\k) \to 0.$$
If $\h_{-1} (\Delta;\k) \neq 0$ then $(x^{\a'} y^\b,i)$ is a critical pair of $I(Y)$ and the conclusion follows from Lemma \ref{Key0}. Thus, we may assume that $\h_{-1}(\Delta;\k) = 0$. Hence, $\h_{-1} (\lk_\Delta x_\gamma;\k) \neq 0$. Thus, we have $\Delta = \langle \{ x_\gamma \} \rangle$. In other words, $x_j, y_\ell \in \sqrt{I(Y):x^{\a'} y^{\b}}$ for all $j\neq \gamma$ and all $\ell$. With an argument similar to that of Step 7 of the proof of Lemma \ref{lem_reg_reduction_to_weight_1}, we deduce a contradiction. Thus, we may assume that $i = 1$. In other words, $\h_0(\Gamma;\k) \neq 0$.

\vskip0.5em \noindent {\bf Step 6.} Reduction to the case $x_\gamma$ is the last variable, i.e., $\gamma = n$. By Step 5, we may assume that $\h_0(\Gamma;\k) \neq 0$. By Lemma \ref{lem_homology_0_bipartite_graph}, $\Gamma = \langle U, V \rangle$ with $U \subseteq X$ and $V \subseteq Y$. If $\h_0(\Delta;\k) \neq 0$ then $(x^{\a'} y^\b,i)$ is a critical pair of $I(Y)$. Hence, the conclusion follows from Lemma \ref{Key0}. Thus, we may assume that $\h_0(\Delta;\k) = 0$. From the long exact sequence
$$0 \to \h_0(\lk_\Delta x_\gamma;\k) \to \h_0(\Gamma;\k) \to \h_0(\Delta;\k) \to  \cdots \to 0,$$
we deduce that $\h_0(\lk_\Delta x_\gamma;\k) \neq 0$. By Lemma \ref{lem_homology_0_bipartite_graph}, we deduce that 
$$\Delta = \Gamma \cup x_\gamma * \langle U_1,V_1 \rangle$$
for some $U_1 \subseteq U$ and $V_1 \subseteq V$. Since $I_\Delta$ corresponds to a Ferrers ideal, we must have $U_1 = U$. Hence, 
$$ \sqrt{I(Y):x^{\a'} y^\b} = (x_j y_k \mid j \in U, k \in V) + x_\gamma (y_k \mid k \in V \setminus V_1) + (x_j \mid j \notin U) + (y_k \mid k \notin V).$$
Assume that $\gamma < n$. Then $x_n \in \sqrt{I(Y):x^{\a'} y^\b}$. In particular, $b_j > \omega$ for some $j \le \lambda_n \le \lambda_\gamma$. By Step 1, this is not possible. Thus, $x_\gamma$ is the last variable.

\vskip0.5em \noindent {\bf Step 7.} Conclusion step. By Step 6, $\gamma = n$. By definition, we deduce that $w(j,\delta) > \omega$ for all $j < n$. In other words, $x^{\a'} y^{\b'} \notin I(Y)$ where $\b'$ is an exponent such that $b'_k = b_k$ if $k \neq \delta$ and $b'_\delta = \omega$. Furthermore, 
$$\sqrt{I(Y) : x^{\a'} y^{\b'}} = \sqrt{I(Y):x^{\a'} y^\b} + (x_\gamma) = \sqrt{J : x^\a y^\b}.$$
In other words, $(x^{\a'} y^{\b'},i)$ is a critical pair of $I(Y)$. By Lemma \ref{Key0}, 
$$\reg S/I(Y) \ge |\a'| + |\b'| + i \ge \omega + \reg S/J.$$
That concludes the proof of the Lemma.
\end{proof}

We are now ready for the main result of the section.

\begin{thm}\label{thm_regularity_tableau} Let $Y$ be a filling of a Young diagram with at least two boxes. Let $(\gamma,\delta)$ be a minimal box of $Y$ with $w(\gamma,\delta) = \omega$. We have 
$$\reg S/I(Y) = \omega - 1 + \max \{\reg S/(I(Y) + (x_{\gamma})), \reg S/(I(Y) + (y_{\delta})) \}.$$    
\end{thm}
\begin{proof} By Lemma \ref{lem_reg_reduction_to_weight_1} and Lemma \ref{lem_reg_lower_bound}, it remains to prove that 
$$\reg S/I(Y) \le \max \{ \reg S/(I(Y) + (x_\gamma)), \reg S/(I(Y) + (y_\delta) ) \}$$
when $\omega = w(\gamma,\delta) = 1$. 
By Lemma \ref{lem_depth_reg_colon}, we may assume that 
$$\reg S/I(Y) = \reg S/I(Y) : x_\gamma + 1 > \reg S/(I(Y) + (x_\gamma)).$$

Let $(x^\a y^\b,i)$ be an extremal pair of $I(Y):x^\gamma$. In particular, $b_\delta = 0$. By Lemma \ref{lem_associated_radicals_edge_weight},
$$\sqrt{(I(Y):x_\gamma) : x^\a y^\b} = \sqrt{(I(Y) + (y_\delta)):(x_\gamma x^\a y^\b)}.$$
Hence, $(x_\gamma x^\a y^\b,i)$ is a critical exponent of $I(Y) + (y_\delta)$. By Lemma \ref{Key0},
$$\reg S/(I(Y) + (y_\delta)) \ge |\a| + |\b| + 1 + i = \reg S/I(Y):x_\gamma + 1.$$
The conclusion follows.
\end{proof}
We now establish Theorem \ref{thm_main}.
\begin{proof}[Proof of Theorem \ref{thm_main}] We prove by induction on the number of boxes of $Y$. If $Y$ has only one box the conclusion is clear. Now assume that $Y$ has at least two boxes and $(\gamma,\delta)$ is a minimal box of $Y$. The shape of $Y$ is $\lambda$ and its conjugate partition is denoted by $\mu$. Let $Y_1$ and $Y_2$ be the tableaux obtained by deleting the $\gamma$th row and $\delta$th column of $Y$ respectively. Also, we let $R_1$ and $R_2$ be polynomial rings on the variables corresponding to $Y_1$ and $Y_2$. By induction, there exist admissible collections of marked boxes $P_1,Q_1$ and $P_2,Q_2$ of $Y_1$ and $Y_2$ respectively such that $\depth R_i/I(Y_i) = d(P_i,Y_i)$ and $\reg R_i/I(Y_i) = r(Q_i,Y_i)$. If $\gamma = 1$ then $V = \{y_{\mu_2+1},\ldots, y_{\mu_1}\}$ is the set of free variable of $I(Y) + (x_\gamma)$. Similarly, if $\delta = 1$ then $U = \{x_{\lambda_2+1},\ldots,x_{\lambda_1}\}$ is the set of free variable of $I(Y) + (y_\delta)$. By Lemma \ref{lem_depth_variables_clearing} and definition, we see that 
\begin{align*}
    \depth S/(I(Y) + (x_\gamma)) & = \depth R_1/I(Y_1) + d_1(M_1)= d(M_1,Y),\\
    \depth S/(I(Y) + (y_\delta)) & =\depth R_2/I(Y_2) + d_1(M_2) = d(M_2,Y),
\end{align*}
where $M_1 = \{ (\gamma,\delta,r)\} \cup P_1$ and $M_2 = \{ (\gamma,\delta,c) \} \cup P_2$. Also, we have $\reg S/(I(Y) + (x_\gamma)) = \reg R_1 / I(Y_1)$ and $\reg S/(I(Y) + (y_\delta)) = \reg R_2/I(Y_2)$. The conclusion follows from induction, Theorem \ref{thm_depth_tableau}, Theorem \ref{thm_regularity_tableau}, and the fact that in any admissible collection of marked boxes of $Y$, the first marked box must be a minimal box of $Y$.    
\end{proof}

\begin{exam}
    Consider the following filling $Y$ of a Young diagram 
    \begin{center}
\begin{tabular}{ccc}
$\begin{ytableau}
*(white) 2 & *(white) 2 & *(white) 5 & *(white) 4\\
*(white) 3 & *(white) 2 & *(white) 4   \\
*(white) 4 & *(white) 6     \\
*(white) 3 
\end{ytableau}$\\[5pt]
\end{tabular}
 \end{center}
Let $I = I(Y)$. Then $M = \{ (1,1,r), (2,2,r), (4,1,r), (3,1,c),(3,2,c)\}$ is a collection of marked boxes such that 
\begin{align*}
\depth S/I(Y) & = d(M,Y) = 3,\\
\reg S/I(Y) &= r(M,Y) = 18.    
\end{align*}
\end{exam}

Finally, we have some concluding remarks. 
\begin{rem}
    \begin{enumerate}
        \item The minimal boxes are not unique. They are incomparable though, namely, if $(\gamma_1,\delta_1)$ and $(\gamma_2,\delta_2)$ are minimal boxes of $Y$ then $(\gamma_1 - \gamma_2) (\delta_1 - \delta_2) < 0$. They will be all be visited no matter what order we choose. For example, in the filling 
    \begin{center}
\begin{tabular}{ccc}
$\begin{ytableau}
*(white) 3 & *(white) 4 & *(yellow) 2 & *(white) 6 & *(white) 7\\
*(white) 4 & *(yellow) 2 & *(white) 4 & *(white) 6   \\
*(yellow) 2 & *(white) 3   \\
*(blue) 2 & *(white) 4\\
*(white) 5
\end{ytableau}$
\end{tabular}
 \end{center}
the yellow boxes are the minimal boxes. The blue box also has the minimum weight but it is not a minimal box.
\item The ordering of the weights on each row (column) matters. For example, let $Y$ be the filling in the previous remark then $\reg S/I(Y) = 27$ and $\depth S/I(Y) = 2$. If we reorder the weight on the first row to $2,3,4,6,7$ then we obtain another filling $Z$ with $\reg S/I(Z) = 25$ and $\depth S/I(Z) = 3$.
       \item The formulae show that the depth and regularity of $I(Y)$ do not depend on the characteristic of the base field $\k$.
       \item In general, the collections of marked boxes that give rise to the depth and regularity of $I(Y)$ are different.
       \item In general, the complexity of computing the depth and computing the regularity of $I(Y)$ are almost identical.
       \item All experimental computations were done in Macaulay2 \cite{M2}.
    \end{enumerate}
\end{rem}

 \section*{Acknowledgments}
 The research of    D. T. Hoang   is     partially supported by  NAFOSTED (Vietnam)   under the grant number 101.04-2023.36.

\end{document}